\documentclass[12pt]{amsart}          %
\usepackage[a4paper]{geometry}	    
\usepackage{amsfonts,amsmath,latexsym,amssymb} 

\usepackage{amscd}

\newtheorem{theorem}{Theorem}[section]

\newtheorem{lemma}{Lemma}[section]

\theoremstyle{definition}
\newtheorem{definition}{Definition}[section]
\newtheorem{example}{Example}[section]
\newtheorem{remark}{Remark}[section]

\numberwithin{equation}{section}

\DeclareMathOperator{\alt}{\vee\hspace{-1.5mm}\wedge\,}
\DeclareMathOperator{\Lim}{Lim}
\newcommand{\RR}{\mathbb R}
\newcommand{\NN}{\mathbb N}
\newcommand{\ZZ}{\mathbb Z}
\begin{document}


\title[Fixed sets and fixed points in  $\Lim$--spaces]{Fixed sets and fixed points for mappings in generalized $\Lim$--spaces of Fr\'echet}
\author{Vladyslav Babenko}
\author{Vira Babenko}
\author{Oleg Kovalenko}
\maketitle

\begin{abstract}
 In this article we discuss a possibility to implement a well-known scheme of proof for contraction mapping theorems in a situation, when convergence, families of  Cauchy sequences, and contractiveness of mappings   are defined axiomatically. We also consider ways to specify families of Cauchy sequences and  contractiveness conditions using distance-like functions with values in some partially ordered set and establish fixed set and point theorems for generalized contractions of the \'Ciri\'c and Caristi types.
\end{abstract}

\maketitle

\section{Introduction}
The theory of fixed point theorems is a well developed domain of Analysis and Topology (see books~\cite{Granas,Agarwal,Kirk} and references therein). In a series of papers (see, for example,~\cite{kasahara,Achari,rus,nieto}) the theory of fixed points was developed in abstract L-spaces in the sense of Fréchet.

The classical contraction mapping theorem in metric spaces, which goes back to Picard, Banach and Caccioppoli, was generalized in many various directions, and in most cases, the following scheme was used to prove the obtained theorems. Consider  a Picard sequence (an orbit of $f$)
 $O(f,x_0):=\{x_0,x_1=f(x_0),\ldots,x_n=f(x_{n-1}),\ldots\,\}$
 and using the fact that  $f$ is contractive in some sense, establish that it is a Cauchy sequence. Using completeness of the space, obtain a point $x$ such that $x_n\to x$ as $n\to\infty.$ Continuity (in some sense) of $f$ gives $f(x)=x$. 
 
The purpose of this note is to discuss the possibility of implementing this scheme in the following abstract situation.
Convergence of sequences $\{x_n\}\subset X$ is defined axiomatically, essentially using an approach that goes back to Fr\'{e}chet~\cite{Frechet} (see also~\cite{Novak}),
but with less requirements on a limiting operator $\Lim$. Continuity (weakly orbital continuity) of a function is then defined as a requirement for operators $\Lim$ and $f$ to commute on orbits of $f$. The family of Cauchy sequences is also defined axiomatically (our approach goes back to~\cite{Irwin}, but with less requirements).
A necessary relation between the family of Cauchy sequences and the $\Lim$  operator is given by the (orbital) completeness requirement --- Cauchy sequences  which are orbit of some $f\colon X\to X$ (actually, even not all -- see e.g.~\cite{Roldan}) must be convergent. A sufficient for fixed point theorems definition of a contractive function is as follows: a function $f$ is locally orbital contractive at a point $x_0\in X$, if the orbit $O(f,x_0)$ is a Cauchy sequence  (this definition will be refined later). Such a definition is motivated by the fact that in the case when the set of Cauchy sequences is determined using a distance function, it is quite consistent with intuitive ideas about contractiveness. Various definitions of contractions can be considered as sufficient conditions for contractiveness in the above presented general sense.
 
In Section~\ref{s::definitions} we give  necessary notations and definitions. There we present the definition and some properties of $\Lim$-spaces, the definition of Cauchy structures, and examples of Cauchy structures defined by a distance-like and a sum-like function taking values in some partially ordered set. An abstract version of a fixed set and point theorem is given in Section~\ref{s::abstractTheorem}.  In Sections~\ref{s::distFuncTheorem}--~\ref{s::otherTheorem2} fixed set and point theorems are presented for above mentioned specific Cauchy structures.

\vspace{0cm}

\section{Notations and definitions}\label{s::definitions}
\subsection{General notations.}
Let $X$ be a nonempty set, $s(X)$ be a set of all sequences $\{ x_n\}$ of elements  from $X$, $s^\infty(X)$ be the set of sequences from $s(X)$ with pair-wise different elements, and $\mathcal{P}(X)$ be the family of all subsets from $X$. 
For $x_1,\ldots,x_n\in X$  by $\langle x_1,\ldots,x_n\rangle$ we  denote the sequence $\{x_1,\ldots,x_n,x_1,\ldots,x_n,\ldots\}$.
For $\{x_n\},\{y_n\}\in s(X)$ we set $\{x_n\}\alt\{y_n\}:=\{ x_1,y_1,x_2,y_2,\ldots,x_n,y_n\ldots\}$.
For a mapping $f\colon X\to X, x_0\in X, n\in\ZZ_+$ set
$
O_n(f,x_0) = \{f^n(x_0), f^{n+1}(x_0),\ldots\}$, $ O(f,x_0)=O_0(f,x_0)=\{ f^n(x_0)\}$, and $O(f)=\{ O_0(f,x_0)\colon x_0\in X\}.$ Finally, let $\mathfrak{N}=\left\{ \{ n_k\}\in s(\NN)\colon n_k\to\infty \;\text{as}\; k\to\infty\right\}$.

\subsection{Limit spaces}
We start with the following definition.
\begin{definition} (cf.~\cite{Frechet,Novak})
A pair $(X,\Lim)$ of a set  $X$ and a mapping $\Lim\colon s(X)\to\mathcal{P}(X)$ such that if $\{{x}_n\}\in s(X)$, then 
$$
\Lim\{{x}_n\}= \Lim\{{x}_{n+1}\},
$$
is called a $\Lim$--space. If for some $\{x_n\}\in s(X)$, $\Lim\, \{x_n\} \neq \emptyset$, then we say that the sequence $\{x_n\}$ is convergent, and write $\{x_n\}\in c(X)$.
\end{definition}

\begin{example}\label{ex::1}
If $X$ is a metric space and for any $\{x_n\}\in s(X)$, $\Lim\,\{x_n\}$ is the set of all partial limits of $\{x_n\}$, then $(X,\Lim)$ is a $\Lim$--space.
\end{example}

\begin{definition} (cf.~\cite{Frechet}). Let a $\Lim$--space $(X,\Lim)$ and $\mathfrak{A}\subset c(X)$
be given. We say that $(X,\Lim)$ is a $\mathfrak{A}$-Fr\'echet space, if for each  $\{x_n\}\in c(X)\cap \mathfrak{A}$ the set $\Lim\,\{x_n\}$ is a singleton. If $\Lim\, \{x_n\}=\{ x\}$, then we say that   $\{x_n\}$ converges to $x$ and  write  $\{x_n\}\to x$ or as $n\to\infty$.
\end{definition}
\begin{definition}
Let a $\Lim$--space $(X,\Lim)$ and a mapping $f\colon X\to X$ be given. We say that $f$ is weakly orbital continuous, if for any $\{x_n\}\in O(f)\cap c(X)$
$$
\Lim \{f(x_n)\}=f(\Lim \{x_n\}).    
$$
\end{definition}
\begin{remark}
The definition of orbital continuity of the mapping $f\colon X\to X$ was introduced in~\cite{ciric}. The definition of weak orbital continuity is less restrictive.
\end{remark}
\subsection{Cauchy structures}
\begin{definition}\label{CS}
A set $\mathfrak{C}\subset s(X)$ is  called a Cauchy structure (we write $\mathfrak{C}\in {CS}$), if 
$\{x_n\}\in \mathfrak{C}\implies \{ x_{n+1}\}\in\mathfrak{C}$.

A Cauchy structure $\mathfrak{C}$ is called a Strong Cauchy structure (we write $\mathfrak{C}\in {SCS}$) if additionally
 $\langle x\rangle\in \mathfrak{C}$ for each $x\in X$, and
 \begin{equation}\label{prop}
 (\langle z_1,\ldots, z_n\rangle\in\mathfrak{C})\implies (z_1=z_n) \text{ for any } z_1\ldots,z_n\in X.
 \end{equation}

A $\Lim$-space $(X,\Lim)$ is called $\mathfrak{C}$-complete if
arbitrary $\{x_n\}\in \mathfrak{C}\cap s^\infty(X)$ is convergent.
\end{definition}
\begin{remark}
Of course, any classical set of Cauchy sequences forms a Cauchy structure.
\end{remark}

\begin{remark}
For generalized metric spaces, the fact that for completeness of a space it is sufficient to require convergence of only Cauchy  sequences with pairwise different elements was noted and used in~\cite{Roldan}.
\end{remark}

\begin{example}
The set $c(X)$ of all convergent sequences in a $\Lim$--space $(X,\Lim)$ can be considered as a Cauchy structure.
\end{example}
\begin{example}
Let a mapping $f\colon X\to X$ be given. The set $\mathfrak{C}=O(f)$ can be considered as a Cauchy structure. In this case $\mathfrak{C}$-complete $\Lim$-space can be called $f$--orbital complete (cf.~\cite{ciric}).
\end{example}

\begin{definition}\label{def::CLim}(cf.~\cite{Irwin})
Setting  $
\mathfrak{C}\Lim\{x_n\}:=\{ x\in X\colon  \langle x\rangle\alt\{x_n\}\in \mathfrak{C}\},
$
we transform a Cauchy space $(X,\mathfrak{C})$ to a $\Lim$--space $(X,\mathfrak{C}\Lim)$.
\end{definition}

If $\mathfrak{C}\in SCS$, and additionally the following condition is satisfied: for all $\{x_n\},\{z_n\}\in s(X)$ and $\{y_n\}\in\mathfrak{C}$, 
\begin{equation}\label{FWcondition}
\{x_n\}\alt\{y_n\},\;\{y_n\}\alt\{z_n\}\in\mathfrak{C}\implies \{x_n\}\alt\{z_n\}\in\mathfrak{C},
\end{equation} 
then, as it is easily seen, the space $(X,\mathfrak{C}\Lim)$ is a $\mathfrak{C}$-Fr\'echet space. Condition~\eqref{FWcondition} is a generalization of well-known Fr\'{e}chet--Wilson conditions, see e.g.~\cite{Wilson}.

Generalizing the definition of the Cauchy sequence in a uniform space, (see e.~g.~\cite[Chapter~6]{Kelley}) we obtain the following example of a Cauchy structure in the sense of definition~\ref{CS}.

Recall (see e.g.~\cite[Chapter~3.2]{zorich}) that a family $\mathcal{U}$ of non-empty subsets of a set $X$ is called a base in $X$, if for any $U,V\in \mathcal{U}$ there exists $W\in\mathcal{U}$ such that $W\subset U\cap V$.

\begin{example}\label{ex::uniformStructure}
 Let $X$ be a nonempty set and $\mathcal{U}$ be a base in $X^2$. The class of all sequences $\{ x_n\}\in s(X)$ such that for any $U\in \mathcal{U}$ there exists $N\in\NN$ with the property $m,n\ge N\implies (x_m,x_n)\in U$, forms a Cauchy structure  $\mathfrak{C}_{\mathcal{U}}$. It is clear that
\begin{multline*}
  \mathfrak{C}_{\mathcal{U}}\Lim\{ x_n\}=\{ x\in X\colon \{ x_n\}\in \mathfrak{C}_{\mathcal{U}} \text{ and } \forall U\in\mathcal{U} \;\exists N\in\NN 
  \\
  \text{ such that }(x_n,x),(x,x_n)\in U \text{ for any } n\ge N\}.  
\end{multline*}
\end{example}

\subsection{Cauchy structures defined by distance functions}
\noindent Let $X$ be a set and $(Y, \Lim_Y)$ be a $\Lim$--space.
\begin{definition}
A mapping $d\colon X^2\to Y$ is called an $Y$-valued distance, if for all $x,y\in X$
\begin{equation}\label{distance}
    d(x,y)=d(y,x)=d(x,x)=d(y,y)\implies x=y.
\end{equation}
\end{definition}

\begin{definition}
The family $\mathfrak{C}_d$ that consists of all sequences $\{x_n\}\in s(X)$ such that $$\{d(x_{m_k},x_{n_k})\}\in c(Y)$$ for any sequences $\{ m_k\},\{ n_k\}\in \mathfrak{N}$, and $\Lim_Y\{d(x_{m_k},x_{n_k})\}$ does not depend on $\{ m_k\}$ and $\{ n_k\}$ is a Cauchy structure.
\end{definition}

\begin{lemma}
Let $(Y, \Lim_Y)$ be such that $\langle \alpha\rangle, \langle \beta\rangle\in c(Y)$ for any $\alpha,\beta\in Y$, and
\begin{equation}\label{add_prop}
\Lim_Y\langle \alpha\rangle=\Lim_Y\langle \beta\rangle \implies \alpha=\beta.
\end{equation}
Then $\mathfrak{C}_d\in SCS$.
\end{lemma}
\begin{proof}
Really, $\langle x\rangle\in \mathfrak{C}_d$, and if $\langle z_1,z_2,\ldots, z_n\rangle\in \mathfrak{C}_d$, then $
\Lim_Y\{ d(z_1,z_n)\}=\Lim_Y \{ d(z_1,z_1)\}=\Lim_Y\{ d(z_n,z_n)\}=\Lim_Y\{ d(z_n,z_1)\}.
$ Due to property~\eqref{add_prop}, we obtain
$
 d(z_1,z_n)= d(z_1,z_1)= d(z_n,z_n)= d(z_n,z_1),
$
and due to~\eqref{distance}, $z_1=z_n$.
Hence  $\mathfrak{C}_d$ satisfies~\eqref{prop}.
\end{proof}
\begin{remark}
The $\Lim$--spaces from Example~\ref{ex::1} satisfy property~\eqref{add_prop}.
\end{remark}

It follows from Definition~\ref{def::CLim} that for any $\{ x_n\}\in\mathfrak{C}_d$
\begin{multline*}
\mathfrak{C}_d\Lim\{ x_n\}\subset\bigcap\limits_{\{ m_k\},\{ n_k\}\in \mathfrak{N}}\{ x\in X\colon \Lim_Y\{ d(x_{m_k},x)\}\\=\Lim_Y\{ d(x,x_{n_k})\}=\Lim_Y\{ d(x_{n_k},x_{m_k})\}=\Lim_Y\langle d(x,x)\rangle\},
\end{multline*}
which is consistent with the definition of convergence in almost all previously considered spaces, in particular in usual metric spaces, in partial metric spaces~\cite{Matthews}, dualistic partial metric spaces~\cite{O'Neill}, dislocated metric spaces~\cite{Hitzler}, metric-like spaces~\cite{Amini}, distance spaces~\cite{Kirk}, metric and distance spaces with more general than $\RR$ sets of values of a distance function 
($K$-metric spaces, cone metric spaces, $M$-distance spaces, probabilistic metric spaces, fuzzy metric spaces, and others -- see~\cite{Babenko_M_dist} and references therein).
        
\subsection{Another way to define Cauchy structures.}
 Let $(Y,\le)$ be a partially ordered set and $\psi\colon\bigcup_{n\in\NN}Y^n\to Y$ be a function such that for any $y_1,y_2,\ldots,y_n\in Y$
\begin{equation}\label{psi}
\psi(y_2,\ldots,y_n)\le\psi(y_1,y_2,\ldots,y_n).
\end{equation}
\begin{example}
In the case $Y=\RR_+$, the function
$
\psi(y_1,y_2,\ldots,y_n)=\sum_{k=1}^ny_k
$
satisfies~\eqref{psi}. Moreover, it is coordinate-wise monotone i.e., 
$$
\psi(y_1,y_2,\ldots,y_n)\leq \psi(z_1,z_2,\ldots,z_n),
$$
provided $y_i\leq z_i$, $i=1,\ldots n$, $n\in\NN$, and for all $y_1,\ldots, y_{m+n}\in Y$,
$$ 
\psi(y_1,\ldots,y_m,y_{m+1},\ldots, y_{m+n})= \psi(y_1,\ldots,y_m,\psi(y_{m+1},\ldots, y_{m+n})).
$$
\end{example}

We say that a set $A\subset Y$ is bounded from above, if there exists $\alpha\in Y$ such that $a\leq \alpha$  for all $a\in A$.  
\begin{definition}
The following set forms a Cauchy structure:
$$
\mathfrak{C}_{\psi} =\left\{ \{ x_n\}\in s(X)\colon \{\psi(x_1,x_2,\ldots,x_n)\}\text{ is bounded from above}\right\}.
$$
\end{definition}

A partial case of the family $\mathfrak{C}_{\psi}$ can be obtained as follows.
\begin{definition}
Let an arbitrary function $d\colon X^2\to Y$ be given. The following set forms a Cauchy structure:
\begin{multline*}
\mathfrak{C}_{\psi,d}=\Big\{ \{ x_n\}\in s(X)\colon \text{the sequence } \\ \big\{ \psi(d(x_1,x_2),\ldots,d(x_{n-1},x_n))\big\}\text{ is bounded from above}\Big\}.
\end{multline*}
\end{definition}

\section{Fixed sets and points theorems}
\subsection{A general fixed set and point theorem}\label{s::abstractTheorem}

\begin{definition}
A set $A\subset X$ is called a fixed set of a mapping $f\colon X\to X$, if $f(A)= A$.  If a singleton $A=\{ x\}$ is a fixed set, then we say that $x$ is a fixed point of the mapping $f$.
\end{definition}

Before stating the theorem, we discuss the question of how  the contractiveness of a mapping could be understood in an abstract situation. Let's start with the observation that contractiveness usually allows one to establish the fact that an orbit $O(f,x_0)$ is a Cauchy sequence.
On the other hand, if the property of being a Cauchy sequence is defined using some distance function, then it means that the elements of the sequence become arbitrarily close as their indices increase. This is quite consistent with intuitive ideas (at least ours) about contractiveness. This is also observed in other cases (see, for example, Example~\ref{ex::uniformStructure}). Therefore, in our opinion, in the abstract situation under consideration, it is natural to adopt the following definition.
\begin{definition}
Let a mapping $f\colon X\to X$ and $x_0\in X$ be such that 
$$
O(f,x_0)\in (\mathfrak{C}\cap s^\infty(X))\cup (s(X)\setminus s^\infty (X)).
$$
Then $f$ will be called locally orbital contractive at the point $x_0$.
\end{definition}

\begin{theorem}\label{th::metaTheorem}
Let a $\Lim$--space $(X,\Lim)$, a set $ \mathfrak{C}\in CS$ be given, and $(X,\Lim)$ be $\mathfrak{C}$-complete. Let also a weakly orbital continuous and locally orbital contractive at some point $x_0\in X$ mapping $f\colon X\to X$  be given. Then $f$ has a fixed set. 

If $\mathfrak{C}\in SCS$, the space $(X,\Lim)$ is a $\mathfrak{C}$-Fr\'echet space,
and $O(f,x_0)\in\mathfrak{C}$, then $f$ has a fixed point. 
If $\{ f^n(x)\} \alt \{ f^n(y)\}\in\mathfrak{C}$ for any $x,y\in X$, then the fixed point is unique. 
\end{theorem}
\begin{proof}
 Let $\mathfrak{C}\in CS$. Assume $\{ f^n(x_0)\}=O(f,x_0)\in \mathfrak{C}\cap s^\infty(X)$. Since the space $(X,\Lim)$ is $\mathfrak{C}$-complete, $\Lim\{ f^n(x_0)\}\neq \emptyset$. Due to weakly orbital continuity of  $f$, we obtain
 \[
     \Lim\{ f^{n}(x_0)\}=\Lim\{ f^{n+1}(x_0)\}=\Lim \{ f(f^{n}(x_0))\}=f(\Lim\{ f^{n}(x_0)\}).
 \]
Thus $\Lim\{ f^{n}(x_0)\}$ is a fixed set of $f$.

If $\{ f^{n}(x_0)\}\in s(X)\setminus s^\infty (X)$, then for some $k<l$ one has $f^k(x_0)=f^l(x_0)$ and hence  $\{ f^k(x_0),f^{k+1}(x_0),\ldots,f^{l-1}(x_0)\}$ is a fixed set.

Let now $\mathfrak{C}\in SCS$, and let $\{ f^n(x_0)\}\in\mathfrak{C}$. If $\{ f^n(x_0)\}\in\mathfrak{C}\cap s^\infty(X)$, then there exits $x\in X$ such that $\Lim\{ f^{n}(x_0)\}=\{ x\}$, and hence $x$ is a fixed point of $f$.
 If $\{ f^n(x_0)\}\in\mathfrak{C}\setminus s^\infty(X)$ and $f^k(x_0)=f^l(x_0), k<l$, then in the case
 $l=k+1$, $f^k(x_0)$ is a fixed point of $f$. 
If $l> k+1$, then the sequence
$
\langle f^k(x_0),f^{k+1}(x_0),\ldots,f^{l-1}(x_0)\rangle 
$
also belongs to $\mathfrak{C}$. By property~\eqref{prop}, $f^l(x_0) = f^{k}(x_0) = f^{l-1}(x_0)$ i.e., $f^{l-1}(x_0)$  is a fixed point of $f$.

Let for any $x,y\in X$ $\{ f^n(x)\} \alt \{ f^n(y)\}\in\mathfrak{C}$. Assume that $x,y$ are fixed points of the function $f$. Using  property~\eqref{prop}, we obtain 
$$
\{ f^n(x)\} \alt \{ f^n(y)\}=\langle x\rangle\alt\langle y\rangle=\langle x,y\rangle \in\mathfrak{C}\implies x=y
$$
i.e., the fixed point is unique. 
\end{proof}

\subsection{A fixed set and point theorem in $\mathfrak{C}_d$-complete $\Lim$-spaces.}\label{s::distFuncTheorem} Let $Y$  be a  $\Lim_Y$--space and at the same time a partially ordered set with a partial order  $\le$. We assume that the partial order and the operator $\Lim_Y$ agree in the following sense: if two sequences $\{\alpha_n\},\{\beta_n\}\in c(Y)$ have equal limits, and $\alpha_n\le\gamma_n\le \beta_n$ for each $n\in\NN$, then $\{\gamma_n\}\in c(Y)$ and
\[
\Lim_Y\{\gamma_n\}=\Lim\{\alpha_n\}=\Lim\{\beta_n\}.
\]
Assume that a $\Lim$-space $(X,\Lim)$ and $Y$--valued distance function  $d$ in $X$ be such that $(X,\Lim)$ is $\mathfrak{C}_d$-complete. 

We say that a set $A\subset X$ is bounded, if
there exist $\alpha,\beta\in Y$ such that $\alpha\leq d(x,y)\leq \beta$ for all $x,y\in A$.

Denote by $\Lambda(Y)$ the set of all  non-decreasing mappings $\lambda\colon Y\to Y$ such that for some non-empty set $\Lim(\lambda)\subset Y$ and arbitrary  $\alpha\in Y$,  $\Lim_Y\{\lambda^n(\alpha)\}=\Lim(\lambda)$. Observe that if $Y$ is an ordinary metric space and $\lambda$ is a contractive in the usual sense mapping, then $\lambda\in \Lambda(Y)$, and $\Lim(\lambda)$ is a singleton.
\begin{theorem}
Assume that for a mapping  $f\colon X\to X$ there exists a point $x_0\in X$ such that the orbit $O(f,x_0)$ is bounded, and two mappings $\lambda_1,\lambda_2\in \Lambda(Y)$ are such that $\Lim(\lambda_1) = \Lim(\lambda_2)$. 
If for arbitrary $x,y\in O_n(f,x_0)$ one can find $x',y',x'',y''\in O_{n-1}(f,x_0)$ 
for which
\begin{equation}\label{contr}
    \lambda_1(d(x',y'))\le d(x,y)\le \lambda_2(d(x'',y'')),
\end{equation}
then the mapping $f$ has a fixed set. If the space $(X,\Lim)$ is $\mathfrak{C}_d$-Fr\'{e}chet, then $f$ has a fixed point.
\end{theorem}

\begin{remark}
Contraction conditions~\eqref{contr} given in this theorem are a generalization of the \'Ciri\'c conditions~\cite{ciric71,Ciric74}.
\end{remark}

\begin{proof} For arbitrary sequences $\{ m_k\},\{ n_k\}$ that tend to infinity, from condition~\eqref{contr} and boundedness of $O(f,x_0)$ we obtain
\[
\lambda_1^{\min (m_k,n_k)}(\alpha)\le d(x_{m_k},x_{n_k})\le \lambda_2^{\min(m_k,n_k)}(\beta),
\]
where $\alpha,\beta\in Y$ are the elements from the definition of boundedness.  Since $\min(m_k,n_k)\to\infty$ whenever $k\to\infty$, we obtain that $\Lim\{d(x_{m_k}, x_{n_k})\} = \Lim(\lambda_1) = \Lim(\lambda_2)$, and hence  $O(f,x_0)\in\mathfrak{C}_d$. Application of Theorem~\ref{th::metaTheorem} finishes the proof.
\end{proof}

\subsection{A fixed set and point theorem in $\mathfrak{C}_{\psi}$--complete $\Lim$-spaces.}\label{s::otherTheorem1}
The following theorem is a variant of the Caristi theorem (see e.g.~\cite{kirk_shahzad} and references therein).
\begin{theorem}
Let $Y$ be a partially ordered set, and a $\Lim$--space $(X,\Lim)$ be  $\mathfrak{C}_{\psi}$-- complete, where $\psi$ is a mapping that satisfies property~\eqref{psi}. Assume that there exists a function $\overline{\psi}\colon\left(\bigcup_n X^n\right)\times Y\to Y$ and such that for  arbitrary $x_1,\ldots, x_{n}\in X$ and $y,z\in Y$,
\begin{equation}\label{psiMajorant}
    \psi(x_1,\ldots, x_n) \leq \overline{\psi}(x_1,\ldots, x_n,y),
\end{equation}
$$ 
\overline{\psi}(x_1,\ldots,x_n, y)\le \overline{\psi}(x_1,\ldots,x_{n-1},\overline{\psi}(x_{n},y)),
$$
and 
$$ 
y\leq z\implies \overline{\psi}(x_1,\ldots,x_n, y)\le \overline{\psi}(x_1,\ldots,x_{n},z). 
$$
If $f\colon X\to X$ is a weakly orbitally continuous function and there exists $\phi\colon X\to Y$ such that for some $x_0\in X$
\begin{equation}\label{caristiContraction}
\overline{\psi}(x,\phi(f(x)))\le \phi(x) \text{ for all } x\in O(f,x_0),
\end{equation}
then $f$ has a fixed set. If the space $(X,\Lim)$ is $\mathfrak{C}_d$-Fr\'{e}chet, then $f$ has a fixed point.
\end{theorem}
\begin{proof}
Let $x_n = f(x_{n-1})$, $n\in\NN$. We prove by induction on $n$ that 
$$
\overline{\psi}(x_{0},x_1,\ldots,x_{n-1},\phi(x_{n}))\leq \phi(x_0).
$$
For $n = 1$ the inequality follows from~\eqref{caristiContraction}. Assume it holds for some $n = k\geq 1$. Then for $n = k+1$, we obtain
\begin{multline*}
 \overline{\psi}(x_{0},x_1,\ldots,x_{k},\phi(x_{k+1}))
 \le 
 \overline{\psi}(x_{0},x_1,\ldots,x_{k-1},\overline{\psi}(x_{k},\phi(x_{k+1})) )
 \\
 \le
 \overline{\psi}(x_{0},x_1,\ldots,x_{k-1},\phi(x_{k})) 
 \le
 \phi(x_0),
\end{multline*}
which finishes the induction step. Inequality~\eqref{psiMajorant} implies that $O(f,x_0)\in \mathfrak{C}_\psi$, and in order to finish the proof, it is sufficient to apply Theorem~\ref{th::metaTheorem}.
\end{proof}

\subsection{Fixed set and point theorems in $\mathfrak{C}_{\psi,d}$--complete $\Lim$-spaces.}\label{s::otherTheorem2}
The following theorem uses a generalization of the \'Ciri\'c-type contraction condition.
\begin{theorem}
Let a $\Lim$--space $(X,\Lim)$ be  $\mathfrak{C}_{\psi,d}$ complete, where $\psi$ is a coordinate-wise non-decreasing mapping that satisfies property~\eqref{psi}, $d\colon X^2\to Y$ is some function, and $Y$ is a partially ordered set. Assume that there exists a non-decreasing $\lambda\colon Y\to Y$ such that for each $y\in Y$ the sequence
\[
\{\psi(y,\lambda(y),\ldots,\lambda^n(y))\}
\]
is bounded from above, $x_0\in X$, and a weakly orbital continuous mapping $f\colon X\to X$ are such that for all 
$n>1$ and all $x,y\in O_n(f,x_0)$ there exist $x',y'\in O_{n-1}(f,x_0)$ that satisfy the inequality
\begin{equation}\label{ciricContraction}
d(x,y)\le \lambda(d(x',y')).
\end{equation}
If there exists $\alpha\in Y$ such that $d(x,y)\leq \alpha$ for all $x,y\in O(f,x_0)$, then  $f$ has a fixed set. If the space $(X,\Lim)$ is $\mathfrak{C}_d$-Fr\'{e}chet, then $f$ has a fixed point.
\end{theorem}
\begin{proof}
Let $x_n = f(x_{n-1})$, $n\in\NN$. Then for each $n\in\NN$, consecutively applying inequality~\eqref{ciricContraction}, we obtain that for some $y,z\in O(f,x_0)$ and $\alpha$ from the statement of the theorem, 
$$
d(x_{n+1},x_n)\leq \lambda^n(d(y,z))\leq \lambda^n(\alpha).
$$
Thus for all $n\in\NN$
$$
\psi(d(x_0,x_1),d(x_1,x_2),\ldots,d(x_n,x_{n+1}))\le\psi(\alpha,\lambda(\alpha),\ldots,\lambda^n(\alpha)),
$$
which implies boudedness of the sequence $\{\psi(d(x_0,x_1),\ldots,d(x_n,x_{n+1}))\}$. Hence $O(f,x_0) \in \mathfrak{C}_{\psi,d}$, and it is enough to apply Theorem~\ref{th::metaTheorem} in order to finish the proof.
\end{proof}

\bibliographystyle{elsarticle-num}
\bibliography{bibliography}

\end{document}